\documentclass{article}

\usepackage{amsmath,amssymb,amsthm}
\usepackage{mathrsfs}
\usepackage{geometry}
\usepackage{color,enumitem,graphicx,cite,mathrsfs,tikz,colortbl}
\usepackage[colorlinks=true,urlcolor=blue,
citecolor=red,linkcolor=blue,linktocpage,pdfpagelabels,
bookmarksnumbered,bookmarksopen]{hyperref}
\usepackage[hyperpageref]{backref}

\newcommand{\eps}{\varepsilon}

\newcommand{\N}{\mathbb N}
\newcommand{\norm}[2][]{\left\|#2\right\|_{#1}}

\newcommand{\pnorm}[2][]{\if #1'' \left|#2\right|_p \else \left|#2\right|_{#1} \fi}

\newcommand{\R}{\mathbb R}

\newcommand{\set}[1]{\left\{#1\right\}}

\newenvironment{enumroman}{\begin{enumerate}
\renewcommand{\theenumi}{$(\roman{enumi})$}
\renewcommand{\labelenumi}{$(\roman{enumi})$}}{\end{enumerate}}

\newtheorem{corollary}{Corollary}[section]
\newtheorem{lemma}[corollary]{Lemma}

\newtheorem{theorem}[corollary]{Theorem}

\theoremstyle{remark}
\newtheorem{remark}[corollary]{Remark}

\numberwithin{equation}{section}

\title{\bf Existence of solutions for  critical Choquard problem with singular coefficients  }

\author{\bf Yang Yang\thanks{Corresponding author. E-mail:yynjnu@126.com. Project supported by NSFC(No. 11501252, No. 11571176).} \\
School of  Science\\
Jiangnan university\\
Wuxi, Jiangsu  214122, China\\
[\bigskipamount]
\bf Yuling Wang\\
School of Science\\
Jiangnan University\\
Wuxi, 214122, China\\
[\bigskipamount]
\bf Yong Wang\\
School of Science\\
Jiangnan University\\
Wuxi, 214122, China}
\date{}

\begin{document}

\maketitle

\begin{abstract}
In this paper, we investigate the following fractional Choquard type equation:
\[
 (- \Delta)_p^s\, u  = \lambda\frac{|u|^{r-2}u}{|x|^\alpha}\,+\gamma \big(\int_\Omega \frac{|u|^q}{|x-y|^\mu}dy\big) |u|^{q-2}u  \ \ \text{in } \Omega,\ \ u = 0 \ \text{in } \R^N \setminus \Omega,
 \]
where $\Omega$ is a bounded domain in $\R^N$ with Lipschitz boundary, $p>1$, $0<s<1$, $N>sp$, $0\leq\alpha\leq sp$, $0<\mu<N$,$\lambda, \gamma>0$, $p\leq r\leq p^*_\alpha$, $p\leq 2q\leq 2p_{\mu,s}^*$, $p_\alpha^*=\frac{(N-\alpha)p}{N-sp}$ and $p_{\mu,s}^*=\frac{(N-\frac{\mu}{2})p}{N-sp}$ are the fractional critical Hardy-Sobolev and the critical exponents in the sense of Hardy-Littlewood-Sobolev inequality, respectively. Under some suitable assumptions,  positive and sign-changing solutions are obtained.
\end{abstract}

\newpage

\section{Introduction and main results}
\noindent

In this paper, we consider the following problem
\begin{equation} \label{1}
\left\{\begin{aligned}
(- \Delta)_p^s\, u &= \lambda\frac{|u|^{r-2}u}{|x|^\alpha}\,+\gamma \big(\int_\Omega \frac{|u|^q}{|x-y|^\mu}dy\big) |u|^{q-2}u,  && \text{in } \Omega,\\[10pt]
u & = 0, && \text{in } \R^N \setminus \Omega,
\end{aligned}\right.
\end{equation}
where $\Omega$ is a bounded domain in $\R^N$ with Lipschitz boundary, $1<p<\infty$, $0<s<1$, $N > sp$, $0\leq \alpha \leq sp$, $p\leq q\leq p_\alpha^*$,
$p\leq 2q\leq 2p_{\mu,s}^*$, and $(- \Delta)_p^s$ is the fractional $p$-Laplacian operator defined on smooth functions by
$$
(- \Delta)_p^s\, u(x) = 2 \lim_{\eps \searrow 0} \int_{\R^N \setminus B_\eps(x)} \frac{|u(x) - u(y)|^{p-2}\, (u(x) - u(y))}{|x - y|^{N+sp}}\, dy, \quad x \in \R^N.
$$

Recent years, problems involving fractional Laplacian and  Choquard equations have been investigated which may be found in \cite{pdSGSM18,VA,ACTY,AGSY,COAABNMY20,COAMY19,GY,GMS,ZSFGMY2,GS5,GSYZ,MS2017,MS20181,MS20182,SGY,CL,PPMXBZ,MLZL2,FWMX4,Y} and references therein.
In \cite{GY}, Gao and Yang established some existence results for critical exponent problem
\begin{equation*}
\left\{\begin{aligned}
-\Delta u&=(\int_{\Omega}\frac{|u|^{2_\mu^*}}{|x-y|^\mu}\,dy)|u|^{2_\mu^*-2}u+\lambda u, \ \ \ \hbox{in} \ \Omega,\\[10pt]
u&=0, \ \ \ \ \  \ \ \ \ \ \ \ \ \hbox{on} \ \partial\Omega,
\end{aligned}\right.
\end{equation*}
where $\Omega$ is a bounded domain of $\R^N$ with Lipschitz boundary, $\lambda$ is a real parameter, $N \geq3$, $2_\mu^*=\frac{2N-\mu}{N-2}$ is the critical exponent in the sense of the Hardy-Littlewood-Sobolev inequality.
Mukherjee and Sreenadh  \cite{TMKS3} extended the results above to  the nonlocal problem and obtained some existence, multiplicity, regularity and nonexistence results of solutions for the following problem
\begin{equation*}
\left\{\begin{aligned}
(-\Delta)^su&=(\int_{\Omega}\frac{|u|^{2_{\mu,s}^*}}{|x-y|^\mu}\,dy)|u|^{2_{\mu,s}^*-2}u+\lambda u, \ \ \  \ \ \hbox{in} \ \Omega,\\[10pt]
u&=0, \ \ \ \ \  \ \ \ \ \ \ \ \ \ \hbox{in}\ \R^N\setminus\Omega,
\end{aligned}\right.
\end{equation*}
where  $N>2s$ and $2_{\mu,s}^*=\frac{2N-\mu} {N-2s}$.
Wang and Yang \cite{YWYY} proved the bifurcation results for the critical Choquard problem
\begin{equation*}
\left\{\begin{aligned}
(-\Delta)^s_pu&=\lambda|u|^{p-2}u+(\int_{\Omega}\frac{|u|^{p_{\mu,s}^*}}{|x-y|^\mu}\,dy)|u|^{p_{\mu,s}^*-2}u, \ \ \  \ \ \hbox{in} \ \Omega,\\[10pt]
u&=0, \ \ \ \ \  \ \ \ \ \ \ \ \ \ \hbox{in} \ \R^N\setminus\Omega,
\end{aligned}\right.
\end{equation*}where  $N>sp$ and $p_{\mu,s}^*=\frac{(N-\frac{\mu}{2})p}{N-sp}$.

For the Hardy-Sobolev exponents involving the fractional $p$ laplacian operator,
Chen, Mosconi and Squassina \cite{MR77777}  studied  the following problem:
\begin{equation}\label{2}
\left\{\begin{aligned}
\left(- \Delta\right)^s_p u & = \frac{|u|^{p^*_\alpha-2} u}{|x|^\alpha}+\mu |u|^{q-2} u,   && \text{in } \Omega;\\[10pt]
u & = 0, &&  \text{in } \mathbb{R}^N\setminus\Omega;
\end{aligned}\right.
\end{equation}
where $p\leq q<\frac{Np}{N-ps}$, and established the existence of positive and sign-changing least energy solutions for problem \eqref{2} by finding the minimizer of the corresponding energy functional on positive Nehari and sign-changing Nehari sets. Yang \cite{MR88888}  studied  the existence of problem \eqref{2} when $q=p$ and obtained the multiplicity and bifurcation results by cohomological  index and pseudo-index. Chen \cite{C} established the existence of positive solutions to the fractional p-Kirchhoff type problem with a generalized Choquard nonlinearity and a critical Hardy-Sobolev term.
In this paper, we shall be interested in the study of Mountain-pass solutions and least energy sign-changing solutions to \eqref{1} and extend the results in \cite{MR77777} to Choquard type.

Now we introduce a variational setting for problem \eqref{1}.  Let
$$
[u]_{s,p} = \left(\int_{\R^{2N}} \frac{|u(x) - u(y)|^p}{|x - y|^{N+sp}}\, dx dy\right)^{1/p}
$$
be the Gagliardo seminorm of the measurable function $u : \R^N \to \R$, and let
$$
W^{s,p}(\R^N) = \set{u \in L^p(\R^N) : [u]_{s,p} < \infty}
$$
be the fractional Sobolev space endowed with the norm
$$
\norm[s,p]{u} = \big(\pnorm[p]{u}^p + [u]_{s,p}^p\big)^{1/p},
$$
where $\pnorm[p]{\cdot}$ is the norm in $L^p(\R^N)$ (see Di Nezza et al.\! \cite{MR713209} for details). We work in the closed linear subspace
$$
W_{0}^{s,p}(\Omega) = \set{u \in W^{s,p}(\R^N) : u = 0 \text{ a.e.\! in } \R^N \setminus \Omega},
$$
equivalently renormed by setting $\norm{\cdot}=[\cdot]_{s,p}$. Moreover,  $p_\alpha^*=\frac{(N-\alpha)p}{N-ps}$ is the fractional critical Hardy-Sobolev exponent, which arises from the general fractional Hardy-Sobolev inequality
\begin{equation*}
\left(\int_{\R^N}\frac{|u|^{p^*_\alpha}}{|x|^\alpha}dx\right)^{\frac{1}{p^*_\alpha}}\leq C(N,p,\alpha)[u]_{s,p}.
\end{equation*}
The latter is the scale invariant inequality and as such is critical for the embedding
\begin{equation*}
W_0^{s,p}(\Omega)\hookrightarrow L^r\left(\Omega,\frac{dx}{|x|^\alpha}\right),
\end{equation*}
in the sense that the latter is continuous for any $r\in [1,p^*_\alpha]$ but (as long as $0 \in \Omega$, as we are assuming) is compact if and only if $r<p^*_\alpha$, and we can set \begin{equation*}S_{\alpha,r}=\inf_{u\in W_0^{s,p}(\Omega)\setminus\{0\}}\frac{\norm{u}^r}{\int_\Omega\frac{|u|^r}{|x|^{\alpha}}dx}.\end{equation*}
 To the Choquard  term,  we have the following Lemma:
\begin{lemma}\cite[Theorem 4.3]{LL}\label{1.4}
Assume $1<r$, $t<\infty$ and $0<\mu<N$ with $\frac{1}{r}+\frac{1}{t}+\frac{\mu}{N}=2$. If $u\in L^r(\R^N)$ and $v\in L^t(\R^N)$,  there exists $C(N,\mu,r,t)>0$ such that
$$
\int_{\R^N}\int_{\R^N}\frac{|u(x)|\,|v(y)|}{|x-y|^{\mu}}\,dxdy\leq C(N,\mu,r,t)\pnorm[r]{u}\pnorm[t]{v}.
$$
\end{lemma}
\noindent By the Hardy-Littlewood-Sobolev inequality, there exists $\widetilde{C}(N,\mu, s, p)>0$ such that
\begin{equation*}\label{1.8}
\int_{\R^N}\int_{\R^N}\frac{|u(x)|^{p_{\mu,s}^\ast}\,|u(y)|^{p_{\mu,s}^\ast}}{|x-y|^{\mu}}\,dxdy\leq \widetilde{C}(N,\mu, s, p)\pnorm[{p_s^{\ast}}]{u}^{2p_{\mu,s}^{\ast}},
\end{equation*}
for all $u\in W_0^{s,p}(\Omega)$. Hence, by the Sobolev embedding,
$$
\int_{\Omega}\int_{\Omega}\frac{|u(x)|^{p_{\mu,s}^\ast}\,|u(y)|^{p_{\mu,s}^\ast}}{|x-y|^{\mu}}\,dxdy\leq C(N,\mu,s,p)\norm{u}^{2p_{\mu,s}^{\ast}},
$$
for some $C(N,\mu,s,p)>0$. Similarity, we define
\begin{equation}\label{1.9}
S_{H,L} := \inf_{u \in W_0^{s,p}(\Omega) \setminus \set{0}}\, \frac{\norm{u}^{p}}{\left(\int_\Omega\int_\Omega\frac{|u(x)|^{p_{\mu,s}^{\ast}}\,|u(y)|^{p_{\mu,s}^{\ast}}\,}
{|x-y|^{\mu}}dxdy\right)^{\frac{p}{2p_{\mu,s}^{\ast}}}}.
\end{equation}
Clearly, $S_{H,L}>0$.
Define
\begin{equation}\label{2222}
S_\mu:=\inf_{u \in W_0^{s,p}(\Omega) \setminus \set{0}}\frac{\norm{u}^p}{\int_\Omega\int_\Omega\frac{|u(x)|^{\frac{p}{2}}\,|u(y)|^{\frac{p}{2}}\,}
{|x-y|^{\mu}}dxdy}.
\end{equation}
So any solution of \eqref{1} is a critical point of the energy functional $J(u):W_0^{s,p}(\Omega)\rightarrow \R$ defined by
$$
J(u)=\frac{1}{p}\norm{u}^p-\frac{\lambda}{r}\int_{\Omega}\frac{|u|^r}{|x|^\alpha}\,dx-\frac{\gamma}{2q}\int_{\Omega}\int_{\Omega}\frac{|u(x)|^q|u(y)|^q}{|x-y|^\mu}\,dxdy.
$$
From \cite[Lemma 2.2]{MR77777}, for any $\lambda,\gamma\in\R$, $1\leq r\leq p^*_\alpha$ and $1\leq q\leq p_{\mu,s}^*$, $J\in C^1$ and $J':W_0^{s,p}(\Omega)\rightarrow W^{-s,p'}(\Omega)$ is both a strong-to-strong and weak-to-weak continuous operator. Now we can state our results as follows.

\begin{theorem} \label{Theorem 2}
Suppose that $0<\mu<N$,
\begin{enumroman}
\item
\begin{equation}\label{11112}
0\leq \alpha <sp,
\left\{\begin{aligned}
&\lambda>0, \gamma>0,\ \ \ if\ p<2q<2p_{\mu,s}^*,\ p< r<p^*_\alpha, \\
&0<\lambda<S_{\alpha,p}, \gamma>0,\ \ if \ p<2q<2p_{\mu,s}^*, \ r=p<p^*_\alpha,\\
&\lambda>0, 0<\gamma<S_\mu, \ \ if \ p=2q<2p_{\mu,s}^*,\ p<r<p^*_\alpha,
\end{aligned}\right.
\end{equation} problem \eqref{1} has a positive solution.
\item
\begin{equation}\label{11111}
0\leq\alpha <sp, \ \ p<q<p_{\mu,s}^*,\ \
\left\{\begin{aligned}
&\lambda>0, \gamma>0,\ \ \ if \ p< r<p^*_\alpha, \\
&0<\lambda<S_{\alpha,p}, \gamma>0,\ \ if  \ r=p<p^*_\alpha,
\end{aligned}\right.
\end{equation}
 problem \eqref{1} has a sign-changing solution.
\end{enumroman}
\end{theorem}
\begin{theorem}
Suppose that $0< \mu <N$, $0\leq \alpha <sp$,  $r<p<p_\alpha^*$, $p<2q=2p_{\mu,s}^*$, $\gamma>0$. Then there exists $\lambda^*>0$ such that problem \eqref{1} has a positive  solution of minimal energy for all $\lambda\geq\lambda^*$.
\end{theorem}

\begin{theorem}
Suppose that $0< \mu <N$, $\alpha=sp$,  $r=p$, $\gamma>0$, $0<\lambda<S_{sp,p}$, problem  \eqref{1} has a positive  solution if $p<2q<2p_{\mu,s}^*$, and a sign-changing solution if $p<q<p_{\mu,s}^*$.
\end{theorem}

\section{Preliminary Results}

\begin{lemma}\label{1444}
Let $0<\mu<N$, $0\leq \alpha\leq sp$, $p\leq 2q\leq 2p_{\mu,s}^*$, $p\leq r\leq p^*_\alpha$. For any $\lambda$, $\gamma$ satisfying,
\begin{equation} \label{111}
\left\{\begin{aligned}
&\lambda>0, \gamma>0,\ \ \ if\ 2q>p, \ r>p, \\
&0<\lambda<S_{\alpha,p}, \gamma>0, \ \ \   if \ 2q>p, \ r=p,\\
&\lambda>0, 0<\gamma<S_\mu, \ \ \   if \ 2q=p, \ r>p,
\end{aligned}\right.
\end{equation}
there exists $\delta_0=\delta_0(\Omega,N,p,s,\mu)>0$, such that for any $u\in W_0^{s,p}(\Omega)$, it holds
$$
\langle J'(u),u\rangle \leq 0 \Rightarrow \norm{u}\geq \delta_0.
$$
\end{lemma}
\begin{proof}
Apply the H$\ddot{o}$lder inequality on the last two terms of
$$
\langle J'(u),u\rangle=\norm{u}^p-\lambda\int_{\Omega}\frac{|u|^r}{|x|^\alpha}\,dx-\gamma\int_{\Omega}\int_{\Omega}\frac{|u(y)|^q|u(x)|^q}{|x-y|^\mu}\,dxdy,
$$
we get
$$\langle J'(u),u\rangle\geq\left\{
\begin{aligned}
&(1-C\norm{u}^{r-p}-C\norm{u}^{2q-p})\norm{u}^p, \ \ if\ 2q>p, \ r>p, \\
&(1-\frac{\lambda}{S_{\alpha,p}}-C\norm{u}^{2q-p})\norm{u}^p,\ \ \ if\ 2q>p, \ r=p,\\
&(1-\frac{\gamma}{S_\mu}-C\norm{u}^{r-p})\norm{u}^p, \ \ \ if \ 2q=p, \ r>p.
\end{aligned}
\right.
$$
The assumption $\langle J'(u),u\rangle\leq0$ forces the parenthesis above to be non-positive, which provides the claimed lower bound.
\end{proof}
\begin{lemma}(Palais-Smale condition)\label{bbb} Suppose  $0<\mu<N$,
\begin{enumroman}
\item If $0\leq \alpha <sp$, $r\leq p<p_\alpha^*$, $q<p_{\mu,s}^*$ and
\begin{equation}\label{227}
\left\{\begin{aligned}
&\lambda>0, \ \gamma>0,\ \ \ if\ 2q> p, \ r>p, \\
&0<\lambda<S_{\alpha,p}, \ \gamma>0, \ \ if \ 2q>p, \ r=p,\\
&\lambda>0,\ 0<\gamma<S_\mu, \ \ \ if \ 2q=p,\ r>p,
\end{aligned}\right.
\end{equation}
$J$ satisfies $(PS)_c$ for all $c\in\R$.
\item If $\alpha=sp$,  $r=p$, $p<2q<2p_{\mu,s}^*$, then for any $0<\lambda<S_{sp,p}$, $\gamma>0$, $J$ satisfies $(PS)_c$ for all $c\in \R$.
\item If $0\leq \alpha <sp,$ $2q=2p_{\mu,s}^*>p$ and $p<r<p^*_\alpha$, then for any $\lambda>0$, $\gamma>0$, $J$ satisfies $(PS)_c$ for all
$$
 c<(\frac{1}{p}-\frac{1}{2p_{\mu,s}^*})\frac{S_{H,L}^{2p_{\mu,s}^*/(2p_{\mu,s}^*-p)}}{\gamma^{p/(2p_{\mu,s}^*-p)}}.
$$

\end{enumroman}
\end{lemma}
\begin{proof}
Suppose that $\{u_k\}$ is a $(PS)_c$ sequence of $J$, that is
$$
J(u_k)=c+o_k(1),\ \ \ \langle J'(u_k),\varphi\rangle=\langle \omega_k,\varphi\rangle, \ \  \omega_k\rightarrow 0 \  in \ W^{-s,p'}(\Omega).
$$
We can obtain
\begin{equation}\label{444}
\norm{u_k}^p-\frac{\lambda p}{r}\int_{\Omega}\frac{|u_k|^r}{|x|^\alpha}\,dx-\frac{\gamma p}{2q}\int_{\Omega}\int_{\Omega}\frac{|u_k(x)|^q|u_k(y)|^q}{|x-y|^\mu}\,dxdy=pJ(u_k)=pc+o_k(1),
\end{equation}
\begin{equation}\label{555}
\norm{u_k}^p-\lambda\int_{\Omega}\frac{|u_k|^r}{|x|^\alpha}\,dx-\gamma\int_{\Omega}\int_{\Omega}\frac{|u_k(x)|^q|u_k(y)|^q}{|x-y|^\mu}\,dxdy=\langle J'(u_k),u_k\rangle=o_k(1)\norm{u_k}.
\end{equation}
So
\begin{equation}\label{333}
C+o_k(1)\norm{u_k}\geq pJ(u_k)-\langle J'(u_k),u_k\rangle=\gamma(1-\frac{p}{2q})\int_{\Omega}\int_{\Omega}\frac{|u_k(y)|^q|u_k(x)|^q}{|x-y|^\mu}\,dxdy
+\lambda(1-\frac{p}{r})\int_{\Omega}\frac{|u_k|^r}{|x|^\alpha}dx.
\end{equation}
First we show that $\{u_k\}$ is bounded in $W_0^{s,p}(\Omega)$.\\
If $r>p$ and $2q>p$,  \eqref{333} implies
$$
\int_{\Omega}\int_{\Omega}\frac{|u_k(y)|^q|u_k(x)|^q}{|x-y|^\mu}\,dxdy\leq C(1+\norm{u_k}), \ \ \int_{\Omega}\frac{|u_k|^r}{|x|^\alpha}dx\leq C(1+\norm{u_k}).
$$
From \eqref{444}, we can get that
\begin{equation}\label{33332}
\  \  \  \  \  \   \  \  \  \ \  \  \   \norm{u_k}^p\leq C(1+\norm{u_k})
\end{equation}
and the boundness of $\{u_k\}$ in $W_0^{s,p}(\Omega)$ readily follows.\\
If $r=p$, $2q>p$, we obtain
\begin{equation}\label{666}
\int_{\Omega}\int_{\Omega}\frac{|u_k(y)|^q|u_k(x)|^q}{|x-y|^\mu}\,dxdy\leq C(1+\norm{u_k}).
\end{equation}
Then by Sobolev inequality, \eqref{555} and \eqref{666}, we have
$$
 \norm{u_k}^p(1-\frac{\lambda}{S_{\alpha,p}})\leq C(1+\norm{u_k}),
$$
so $\{u_k\}$ is bounded in $W_0^{s,p}(\Omega)$ if $\lambda<S_{\alpha,p}$. \\
If $r>p$, $2q=p$, \eqref{333} implies
$$
\int_{\Omega}\frac{|u_k|^r}{|x|^\alpha}dx\leq C(1+\norm{u_k}).
$$
From \eqref{2222} and \eqref{444},
$$
\left(1-\frac{\gamma}{S_\mu}\right)\norm{u_k}^p\leq\norm{u_k}^p-\gamma\int_{\Omega}\int_{\Omega}\frac{|u_k|^{p/2}|u_k|^{p/2}}{|x-y|^\mu}\,dxdy\leq C(1+\norm{u_k}),
$$
implies boundedness of $\{u_k\}$ if $0<\gamma<S_\mu$.\\
Thus, $\{u_k\}$ is bounded and passing if necessary to a subsequence such that $u_k\rightharpoonup u$ in $W_0^{s,p}(\Omega)$ as $k\rightarrow \infty$. From  \cite[Lemma 3.2]{MR874523}, we get
\begin{equation}\label{221}
\ \ \ \ \ \ \ \ \ \ \ \ \  \  \    \norm{u_k}^p=\norm{u_k-u}^p+\norm{u}^p+o(1),
\end{equation}
also \cite[Lemma 2.3]{GFYM123} gives
\begin{equation}\label{223}
 \int_{\Omega}\int_{\Omega}\frac{|u_k(x)|^q|u_k(y)|^q}{|x-y|^{\mu}}\,dxdy= \int_{\Omega}\int_{\Omega}\frac{|u_k(x)-u(x)|^q|u_k(y)-u(y)|^q}{|x-y|^{\mu}}\,dxdy+\int_{\Omega}
 \int_{\Omega}\frac{|u(x)|^q|u(y)|^q}{|x-y|^{\mu}}\,dxdy+o(1),
\end{equation}
when $1\leq q\leq p^*_{\mu,s}$.
In addition, if  $1\leq q<p^*_{\mu,s}$,
$$
\int_{\Omega}\int_{\Omega}\frac{|u_k(x)|^q|u_k(y)|^q}{|x-y|^{\mu}}\,dxdy\rightarrow\int_{\Omega}\int_{\Omega}
\frac{|u(x)|^q|u(y)|^q}{|x-y|^{\mu}}\,dxdy, \ \ as \ \ k\rightarrow\infty.
$$
since
$$
 \int_{\Omega}\int_{\Omega}\frac{|u_k(x)-u(x)|^q|u_k(y)-u(y)|^q}{|x-y|^{\mu}}\,dxdy\leq C(N,\mu)\pnorm[\frac{2N}{2N-\mu}]{(u_k-u)^q}\rightarrow 0, \ \ as\ \ k\rightarrow+\infty.
$$
On the other hand, for $1\leq r<p_\mu^*$,
\begin{equation}\label{Hardy}
\int_\Omega\frac{|u_k(x)|^r}{|x|^\alpha}dx=\int_\Omega\frac{|u(x)|^r}{|x|^\alpha}dx+o(1).\end{equation}
We prove that $u_k\rightarrow u$ in $W_0^{s,p}(\Omega)$. If $0\leq \alpha<sp$,  $q<p_{\mu,s}^*$ and $r<p^*_\alpha$, by the boundedness of $\{u_k\}$, \eqref{221}, \eqref{223} and \eqref{Hardy},  we can get
\begin{equation*}
 o_k(1)=\langle J'(u_k),u_k\rangle-\langle J'(u),u \rangle=\norm{u_k-u}^p,
\end{equation*}
which shows convergence.\\
If $\alpha=sp$, $r=p$ and  $q<p_{\mu,s}^*$, we can obtain
\begin{equation}\label{2.15}
 o_k(1)=\langle J'(u_k),u_k\rangle-\langle J'(u),u \rangle=\norm{u_k-u}^p-\lambda\int_\Omega\frac{|u_k-u|^p}{|x|^{sp}}dx
 \geq(1-\frac{\lambda}{S_{sp,p}})\norm{u_k-u}^p+o_k(1),
\end{equation}
which implies the convergence, if $0<\lambda<S_{sp,p}$.

Now it suffices to analyze the case $p<2q=2p_{\mu,s}^*$, $p<r<p^*_\alpha$. It is clear that $J'(u)=0$ and from $\langle J'(u),u \rangle=0$, we have that
\begin{equation}\label{225}
J(u)=J(u)-\frac{1}{p}\langle J'(u),u \rangle=\lambda(\frac{1}{p}-\frac{1}{r})\int_{\Omega}\frac{|u|^r}{|x|^\alpha}\,dx+\gamma(\frac{1}{p}-\frac{1}{2p_{\mu,s}^*})\int_{\Omega}\int_{\Omega}\frac{|u|^q|u|^q}{|x-y|^\mu}\,dxdy\geq0.
\end{equation}
 Then we get
\begin{equation}\label{226}
J(u_k)=J(u)+\frac{1}{p}\norm{u_k-u}^p-\frac{\gamma}{2p_{\mu,s}^*}\int_{\Omega}\int_{\Omega}\frac{|u_k-u|^{p_{\mu,s}^*}|u_k-u|^{p_{\mu,s}^*}}{|x-y|^\mu}\,dxdy+o_k(1),
\end{equation}
\begin{equation}\label{224}
o_k(1)=\langle J'(u_k),u_k \rangle-\langle J'(u),u \rangle=\norm{u_k-u}^p-\gamma\int_{\Omega}\int_{\Omega}\frac{|u_k-u|^{p_{\mu,s}^*}|u_k-u|^{p_{\mu,s}^*}}{|x-y|^\mu}\,dxdy+o_k(1).
\end{equation}
For $p<r<p_\alpha^*$ and $c<(\frac{1}{p}-\frac{1}{2p_{\mu,s}^*})\frac{S_{H,L}^{2p_{\mu,s}^*/(2p_{\mu,s}^*-p)}}{\gamma^{p/(2p_{\mu,s}^*-p)}}$.
By virtue of \eqref{224},
$$
\frac{1}{p}\norm{u_k-u}^p-\frac{\gamma}{2p_{\mu,s}^*}\int_{\Omega}\int_{\Omega}\frac{|u_k(y)-u(y)|^{p_{\mu,s}^*}|u_k(x)-u(x)|^{p_{\mu,s}^*}}{|x-y|^\mu}\,dxdy=(\frac{1}{p}-\frac{1}{2p_{\mu,s}^*})\norm{u_k-u}^p+o_k(1).
$$
From \eqref{225} and \eqref{226} we have
$$
\frac{1}{p}\norm{u_k-u}^p-\frac{\gamma}{2p_{\mu,s}^*}\int_{\Omega}\frac{|u_k(y)-u(y)|^{p_{\mu,s}^*}|u_k(x)-u(x)|^{p_{\mu,s}^*}}{|x-y|^\mu}\,dxdy\leq J(u_k)+o_k(1)=c+o_k(1).
$$
Therefore
$$
\hbox{lim sup}_{k\to\infty}(\frac{1}{p}-\frac{1}{2p_{\mu,s}^*})\norm{u_k-u}^p<(\frac{1}{p}-\frac{1}{2p_{\mu,s}^*})\frac{S_{H,L}^{\frac{2p_{\mu,s}^*}{2p_{\mu,s}^*-p}}}
{\gamma^{{\frac{p}{2p_{\mu,s}^*-p}}}}.
$$
Using this and  \eqref{1.9}, we have
\begin{eqnarray*}
\begin{aligned}
&o_k(1)=\norm{u_k-u}^p-\gamma\int_{\Omega}\int_{\Omega}\frac{|u_k(y)-u(y)|^{p_{\mu,s}^*}|u_k(x)-u(x)|^{p_{\mu,s}^*}}{|x-y|^\mu}\,dxdy\geq \norm{u_k-u}^p-\gamma S_{H,L}^{-\frac{2p_{\mu,s}^*}{p}}\norm{u_k-u}^{2p_{\mu,s}^*}
\\&=\norm{u_k-u}^p(1-\gamma S_{H,L}^{-\frac{2p_{\mu,s}^*}{p}}\norm{u_k-u}^{2p_{\mu,s}^*-p})\geq \omega_1\norm{u_k-u}^p,
\end{aligned}
\end{eqnarray*}
for some $\omega_1>0$, giving the claim.
\end{proof}

\begin{remark}\label{219}
Inspecting the proof, we see that the boundedness of Palais-Smale sequence follows solely from the condition $|J(u_k)|\leq c$ and $|\langle J'(u_k),u_k  \rangle|\leq c\norm{u_k}$.
\end{remark}

\section{Positive solution}
Existence of a nontrivial solution follows from the standard Mountain pass approach. Let
\begin{equation*}
\Gamma=\{\gamma\in C^0([0,1];W_0^{s,p}(\Omega)):\gamma(0)=0, J(\gamma(1))<0\},
\end{equation*}
\begin{equation*}\label{2223}
c_1=\mathop{inf}_{\gamma\in\Gamma}\mathop{sup}_{t\in[0,1]}J(\gamma(t)).
\end{equation*}
We have the following Theorem:
\begin{theorem}\label{mountain} Problem \eqref{1} has a nontrivial solution $u$ which satisfies $J(u)=c_1$ if one of the following conditions holds.
\begin{enumroman}
\item
\begin{equation}\label{227}
0\leq \alpha <sp,
\left\{\begin{aligned}
&\lambda>0, \ \gamma>0,\ \ \ if\  p< 2q< 2p_{\mu,s}^*, \ p< r <p_\alpha^*, \\
&0<\lambda<S_{\alpha,p}, \ \gamma>0, \ \ if\   p< 2q< 2p_{\mu,s}^*, \ p=r <p_\alpha^*,\\
&\lambda>0,\ 0<\gamma<S_\mu, \ \ \ if \ p= 2q< 2p_{\mu,s}^*, \ p<r <p_\alpha^*,
\end{aligned}\right.
\end{equation}
\item \begin{equation}\label{A}\alpha=sp,\ \ \  r=p,\ \ \  p<2q<2p_{\mu,s}^*,\ \ \  0<\lambda<S_{sp,p},\end{equation}
\item
\begin{equation}\label{B}0\leq \alpha <sp,\ \ \ p<r<p_\alpha^*,\ \ \ p<2q=2p_{\mu,s}^*, \ \lambda \ is \ large \ enough.\end{equation}
\end{enumroman}
\end{theorem}
\begin{proof}
From $max\{2q,r\}>p$, we can obtain that for any given $u\in W_0^{s,p}(\Omega)$, $J(tu)\rightarrow-\infty$ for $t\rightarrow+\infty$. From Lemma \ref{1444}, we can get that
$$J(u)\geq\left\{
\begin{aligned}
&(\frac{1}{p}-C\norm{u}^{r-p}-C\norm{u}^{2q-p})\norm{u}^p,\ \ if\ 2q> p,\ r>p,\\
&(\frac{1}{p}-\frac{\lambda}{pS_{\alpha,p}}-C\norm{u}^{2q-p})\norm{u}^p, \ \ if \ 2q>p, \ r=p,\\
&(\frac{1}{p}-\frac{\gamma}{pS_\mu}-C\norm{u}^{r-p})\norm{u}^p, \ \ \ if \ 2q=p, \ r>p.
\end{aligned}
\right.
$$
Under the assumption  \eqref{227}, \eqref{A} or \eqref{B}, we can imply that
$$
 \mathop{inf}_{\norm[s,p]{u}=\varrho}J(u)>0,
$$
for sufficiently small $\varrho>0$. Therefore $c_1>0$ and  to apply the Mountain pass theorem it only needs to show that J satisfies the $(PS)_{c_1}$ condition. By Lemma \ref{bbb} this is certainly true for case (i) and  case (ii).\\
Now, it suffices to consider the  case (iii),
\begin{eqnarray}\label{2210}
\begin{aligned}
c_1&\leq \mathop{sup}_{t\geq0}J(tu)=\mathop{sup}_{t\geq 0}\left[\frac{t^p}{p}\norm{u}^p-\frac{\lambda t^r}{r}\int_{\Omega}\frac{|u|^r}{|x|^\alpha}\,dx-\frac{\gamma t^{2p_{\mu,s}^*}}{2p_{\mu,s}^*}\int_{\Omega}\int_{\Omega}\frac{|u|^{p_{\mu,s}^*}|u|^{p_{\mu,s}^*}}{|x-y|^\mu}\,dxdy\right]
\\&\leq \mathop{sup}_{t\geq0}\left[\frac{t^p}{p}\norm{u}^p-\frac{\lambda t^r}{r}\int_{\Omega}\frac{|u|^r}{|x|^\alpha}\,dx\right]
\\&=(\frac{1}{p}-\frac{1}{r})\frac{\norm{u}^{\frac{rp}{r-p}}}{(\lambda\int_{\Omega}\frac{|u|^r}{|x|^\alpha}\,dx)^{p/(r-p)}}
<(\frac{1}{p}-\frac{1}{2p_{\mu,s}^*})\frac{S_{H,L}^{\frac{p_{\mu,s}^*}{2p_{\mu,s}^*-p}}}{\gamma^\frac{p}{2p_{\mu,s}^*-p}}
\end{aligned}
\end{eqnarray}
if $\lambda>0$ is sufficiently large.
\end{proof}

Now define the Nehari manifold associated to $J$ as
$$
 \mathscr{N}=\{u\in W_0^{s,p}(\Omega)\setminus\{0\}:\langle J'(u),u \rangle=0\}
$$
with subsets
$$
\mathscr{N}_+=\mathscr{N}\cap\{u\geq 0\},  \  \ \ \ \  \mathscr{N}_-=\mathscr{N}\cap\{u\leq 0\}.
$$
Let $u^+=$max$\{0,u\}$ and $u^-=$min$\{u,0\}$.
\begin{lemma}\label{122}(\cite[Lemma 2.5]{MR77777})
For any $u\in W_0^{s,p}{(\Omega)}$, it holds
\begin{equation}\label{1333}
\langle(-\Delta)_p^su^\pm,u^\pm\rangle\leq\langle(-\Delta)_p^su,u^\pm\rangle\leq\langle(-\Delta)_p^su,u\rangle
\end{equation}
with strict inequality as long as $u$ is sign-changing.
\end{lemma}
\begin{theorem}
Under the assumptions in Theorem \ref{mountain}, there exists a nonnegative critical point $\omega$ solving $J(\omega)=c_1$ and
\begin{equation}\label{33335}
c_1=\mathop{inf}_{u\in \mathscr{N}_+}J(u).
\end{equation}
\end{theorem}
\begin{proof}
The existence of a critical point at level $c_1$ is obtained in Theorem \ref{mountain} , so  it only remains to show that $\omega\in \mathscr{N}_+$ and $\eqref{33335}$. Fix $u\neq 0$ and define the function
\[
\R_+\ni s \mapsto \psi(s):=J(s^{1/p}u)=\frac{s}{p}\norm{u}^p-\frac{s^{r/p}\lambda}{r}\int_{\Omega}\frac{|u|^r}{|x|^\alpha}\,dx-\frac{s^{2q/p}\gamma}{2q}\int_{\Omega}\int_{\Omega}\frac{|u(x)|^q|u(y)|^q}{|x-y|^\mu}\,dxdy.
\]
It is clear that
\[
\psi(0)=0, \ \ \ \hbox{lim}_{s\rightarrow+\infty}\psi(s)=-\infty, \ \ \ \psi \ is \ concave.
\]
Moreover, from the assumption $max\{2q,r\}>p$, $\psi$ is strictly concave. Therefore $\psi$ has a unique maximum $s_u$, which is positive due to
$$\psi'(0)=\left\{
\begin{aligned}
&\frac{1}{p}\norm{u}^p, \ \ \ \  2q>p, \ \ r>p,
\\&\frac{1}{p}\left(\norm{u}^p-\lambda \int_\Omega \frac{|u|^r}{|x|^\alpha}dx\right), \ \ \ 2q>p, \ \ r=p, \ \ 0<\lambda<S_{\alpha,p},
\\&\frac{1}{p}\left(\norm{u}^p-\gamma\int_{\Omega}\int_{\Omega}\frac{|u(x)|^{p/2}|u(y)|^{p/2}}{|x-y|^{\mu}}\,dxdy\right), \ \ \ 2q=p, \ r>p,\ \ 0<\gamma<S_\mu,
\end{aligned}
\right.
$$
all of which are strictly positive. Then $\psi'(s)>0$ for $s<s_u$, $\psi'(s)<0$ for $s>s_u$. Changing variable $s=t^p$, this translates to
\begin{equation}\label{44445}
\left\{\begin{aligned}
&J(tu)\rightarrow -\infty, \ \ for \ t \rightarrow +\infty,
\\& t\mapsto J(tu),  \ \ \ has \ a \  unique \ positive \ maximum \ t_u,
\\& \langle J'(tu),u \rangle(t-t_u)<0, \ \ \forall \ t\neq t_u.
\end{aligned}\right.
\end{equation}
Hence, given $u\in \mathscr{N}$, $J(u)=\mathop{sup}_{t\geq 0}J(tu)$, it implies
$$
c_1\leq \mathop{inf}_{u\in \mathscr{N}}J(u),
$$
since  that
$$
c_1=\mathop{inf}_{\gamma\in\Gamma}\mathop{sup}_{t\in[0,1]}J(\gamma(t))\leq \mathop{inf}_{u\in W_0^{s,p}(\Omega)}\mathop{sup}_{t\geq 0}J(tu).
$$
In addition, the critical point $\omega$ at level $c_1$ certainly lies in $\mathscr{N}$, then we can get that
 $$
c_1= \mathop{inf}_{u\in \mathscr{N}}J(u).
$$ Now it remains to show that
$$
\mathop{inf}_{u\in \mathscr{N}}J(u)=\mathop{inf}_{u\in\mathscr{N}_+}J(u),
$$
i.e. that the ground state $\omega$ solving $J(\omega)=c_1$ can be chosen nonegative. Clearly the inequality $\leq$ above suffices. Since $J$ is even and $\omega\neq 0$, we may suppose without loss of generality that $\omega^+\neq 0$. By \eqref{1333}, we have
$$
\langle J'(\omega^+),\omega^+\rangle\leq \langle J'(\omega),\omega^+ \rangle=0,
$$
so that $t_{\omega^+}$ defined in \eqref{44445} satisfies $t_{\omega^+}\leq 1$. However
\begin{eqnarray*}
\begin{aligned}
\mathop{inf}_{u\in\mathscr{N_+}}J(u)&\leq J(t_{\omega^+}\omega^+)=t_{\omega^+}^r(\frac{\lambda}{p}-\frac{\lambda}{r})\int_{\Omega}\frac{|\omega^+|^r}{|x|^\alpha}\,dx+t_{\omega^+}^{2q}(\frac{\gamma}{p}-\frac{\gamma}{2q})\int_{\Omega}\int_{\Omega}\frac{|\omega^+(x)|^q|\omega^+(y)|^q}{|x-y|^\mu}\,dxdy
\\&\leq(\frac{\lambda}{p}-\frac{\lambda}{r})\int_\Omega\frac{|\omega|^r}{|x|^\alpha}\,dx+(\frac{\gamma}{p}-\frac{\gamma}{2q})\int_{\Omega}\int_{\Omega}\frac{|\omega|^q|\omega|^q}{|x-y|^\mu}\,dxdy
\\&=J(\omega)=c_1=\mathop{inf}_{u\in \mathscr{N}}J(u).
\end{aligned}
\end{eqnarray*}
Finally, observe that $\omega^-\neq 0$ implies that inequality in the second line of the previous chain is strict, therefore the mountain pass solution must be of constant sign.
\end{proof}

\section{Sign-changing solution}
Let
\begin{equation}
\ \ \ \ \ \ \ \ \ \ \ \ \ \ \ \  \mathscr{N}_{sc}=\{u\in\mathscr{N}:u^\pm\neq0, \langle J'(u),u^\pm \rangle=0\}.
\end{equation}
Clearly, any sign-changing solution to \eqref{1} belongs to $\mathscr{N}_{sc}$. First we show that  $\mathscr{N}_{sc}$ is nonempty. This is the content of the following lemma.
\begin{lemma}\label{4.1}
Let $u\in W_0^{s,p}(\Omega)$ be such that $u^\pm\neq0$ and  parameters satisfy  one of the  following conditions:
\begin{enumroman}
\item
\begin{equation}\label{4.111111}
0\leq\alpha<sp, \ \  p<q<p_{\mu,s}^*,\left\{\begin{aligned}
&\lambda>0, \gamma>0,\ \ \ if\  p< r<p^*_\alpha, \\
&0<\lambda<S_{\alpha,p}, \gamma>0,\ \  if  \ p=r<p^*_\alpha.
\end{aligned}\right.
\end{equation}
\item
\begin{equation}\label{4.122222}
\alpha=sp,\ \  r=p,\ \  p<q<p_{\mu,s}^*,\ \  0<\lambda< S_{sp,p}, \ \ \gamma>0.
\end{equation}
\end{enumroman}
Then the maximum
$$
 \mathop{sup}\limits_{(t_1,t_2)\in\R^2}J(t_1u^++t_2u^-)
$$
is attained at a unique $(\bar{t_1},\bar{t_2})\in\R_{+}^2$. Moreover, $(\bar{t_1},\bar{t_2})$ is a global maximum point if and only if
\begin{equation}\label{228}
\ \ \ \ \ \ \  \ \ \  \bar{t_1}\cdot\bar{t_2}>0, \ \ \ \ \ \  \bar{t_1}u^++\bar{t_2}u^-\in\mathscr{N}_{sc},
\end{equation}
and
\begin{equation}\label{229}
\ \ \ \ \ \ \  \ \ \ \ \langle J'(u),u^+ \rangle(\bar{t_1}-1)\geq0, \ \ \  \langle J'(u),u^- \rangle(\bar{t_2}-1)\geq 0.
\end{equation}
\end{lemma}
\begin{proof}
Let $\Omega_{\pm}=supp(u^\pm)$,
\begin{eqnarray*}
\begin{aligned}
J(t_1u^++t_2u^-)&=|t_1|^pA_++|t_2|^pA_{-}-|t_1|^rB_{+}-|t_2|^rB_{-}-|t_1|^{2q}C_{+}-|t_2|^{2q}C_{-}
\\&-\frac{\gamma |t_1|^q|t_2|^q}{q}\int_{\Omega\times\Omega}\frac{|u^+(x)|^q|u^-(y)|^q}{|x-y|^\mu}\,dxdy+\frac{2}{p}\int_{\Omega_+\times\Omega_-}
\frac{|t_1u^+(x)-t_2u^-(y)|^p}{|x-y|^{N+sp}}\,dxdy,
\end{aligned}
\end{eqnarray*}
where
\[
A_{\pm}=\frac{1}{p}\int_{\Omega_{\mp}^c\times\Omega_{\mp}^c}\frac{|u^+(x)-u^+(y)|^p}{|x-y|^{N+sp}}\,dxdy \ \ \ B_{\pm}=\frac{\lambda}{r}\int_{\Omega_{\pm}}\frac{|u|^r}{|x|^\alpha}\,dx \ \ C_{\pm}=\frac{\gamma}{2q}\int_{\Omega_{\pm}}\frac{|u(y)|^q|u(x)|^q}{|x-y|^\mu}\,dxdy
\]
Since $max\{q,r\}>p$ and $u^\pm\neq 0$, we can get that
$$
 \mathop{lim}\limits_{|t_1|+|t_2|\rightarrow+\infty}J(t_1u^++t_2u^-)=-\infty,
$$
so that a maximum exists. Since $u^-\leq 0$, analyzing the mixed integral term shows that the maximum must be attained on $t_1, t_2\geq0$ (or on $t_1,t_2\leq0$), which we will suppose henceforth. Now consider the function
$$
 \R_+^2\ni(s_1,s_2)\mapsto \psi(s_1,s_2)=J(s_1^{1/p}u^++s_2^{1/p}u^-)
$$
Then
\begin{eqnarray*}
\begin{aligned}
\psi(s_1,s_2)&=s_1A_++s_2A_--s_1^{r/p}B_+-s_2^{r/p}B_--s_1^{2q/p}C_+-s_2^{2q/p}C_-
\\&-\frac{\gamma |s_1|^{q/p}|s_2|^{q/p}}{q}\int_{\Omega\times\Omega}\frac{|u^+(x)|^q|u^-(y)|^q}{|x-y|^\mu}\,dxdy+\frac{2}{p}\int_{\Omega_+\times\Omega_-}
\frac{|s_1^{\frac{1}{p}}u^+(x)-s_2^{\frac{1}{p}}u^-(y)|^p}{|x-y|^{N+sp}}
\end{aligned}
\end{eqnarray*}
and
\begin{equation}\label{335}
\begin{aligned}
&\frac{\partial\psi}{\partial s_1}(s_1,s_2)=\frac{s_1^{-1}}{p}\langle J'(s_1^{\frac{1}{p}}u^++s_2^{\frac{1}{p}}u^-),u^+ \rangle,
\\&\frac{\partial\psi}{\partial s_2}(s_1,s_2)=\frac{s_2^{-1}}{p}\langle J'(s_1^{1/p}u^++s_2^{1/p}u^-),u^-\rangle.
\end{aligned}
\end{equation}
 since $min\{q,r\}\geq p$, moreover for any $a,b\geq0$, the function
$$
 \R_{+}^2\ni(s_1,s_2)\mapsto|s_1^{1/p}a+s_2^{1/p}b|^p
$$
is also concave,  Therefore $\psi$ is strictly concave in $\R_+^2$. A direct computation shows that for $(s_1,s_2)\neq(0,0)$
$$\frac{\partial}{\partial s_1}\psi\big|_{s_1=0}=\left\{
\begin{aligned}
&\frac{1}{p}\norm{u^+}^p,\ \ \ \ \ \ \ \ \ \ \ \ \ \ \ \ \ \ \ \ \ \ \ \ \ \ \ \ \ \ \ \ if\  q>p, \ \ r>p,\\
&\frac{\norm{u^+}^p}{p}-\frac{\lambda}{p}\int_\Omega\frac{{|u^+|}^p}{|x|^\alpha}dx, \ \ \ \ \  \ \ \ \ \ \ \ \ \ \ if\   q>p, \ \ r=p,
\end{aligned}
\right.
$$
$$\frac{\partial}{\partial s_2}\psi\big|_{s_2=0}=\left\{
\begin{aligned}
&\frac{1}{p}\norm{u^-}^p,\ \ \ \ \ \ \ \ \ \ \ \ \ \ \ \ \ \ \ \ \ \ \ \ \ \ \ \ \ \ \ \ if\ \ q>p, \ \ r>p, \\
&\frac{\norm{u^-}^p}{p}-\frac{\lambda}{p}\int_\Omega\frac{{|u^-|}^p}{|x|^\alpha}dx, \ \ \ \ \  \ \ \ \ \ \ \ \ \ \ if \ \  q>p, \ \ r=p,
\end{aligned}
\right.
$$
In both cases the derivatives are strictly positive due to \eqref{4.111111} or \eqref{4.122222}, hence the maximum of $\psi$ is attained in the interior of $\R_+^2$ and $\psi$ has its (unique) maximum at $(\bar{s_1},\bar{s_2})$ if and only if $\nabla \psi(\bar{s_1},\bar{s_2})=(0,0)$, which corresponds to the unique maximun for $J(t_1u^++t_2u^-)$ setting $\bar{t_i}=s_i^{-1/p}, i=1,2$. Explicitly computing $\nabla \psi(\bar{s_1},\bar{s_2})=\textbf{0}$ through \eqref{335} gives conditions \eqref{228}.\\
To prove \eqref{229}, it is clear that the concavity of $\psi$ is equivalent to
$$
 (\nabla\psi(\xi)-\nabla\psi(\eta))\cdot(\eta-\xi)\geq0, \ \ \ \forall \xi,\eta\in\R_{+}^2.
$$
Let $\xi=(1,1)$ and  $\eta=(\bar{s_1},1)$ or $\eta=(1,\bar{s_2})$, where $(\bar{s_1},\bar{s_2})$ is the maximum point for $\psi$, we obtain
$$
 \frac{\partial \psi}{\partial s_1}(1,1)(\bar{s_1}-1)\geq0, \ \ \ \ \ \ \ \  \ \ \frac{\partial\psi}{\partial s_2}(1,1)(\bar{s_2}-1)\geq0.
$$
Use \eqref{335} in these relations, we can get  \eqref{229}, since $\bar{s_1}-1$ and $\bar{s_1}^{1/p}-1$ have the same sign.
\end{proof}

\begin{theorem}\label{theorem4.2}
Under the conditions in Theorem \ref{4.1}, the problem
\begin{equation}\label{336}
c_2:=\mathop{inf}\limits_{u\in \mathscr{N}_{sc}}J(u)=\mathop{inf}\limits_{u^\pm\neq0}\mathop{sup}\limits_{(t_1,t_2)\in\R^2}J(t_1u^++t_2u^-)
\end{equation}
has a solution $v$ which is a sign-changing critical point for $J$.
\end{theorem}
\begin{proof} We only give the proof for case (i). In fact,
the second equality in \eqref{336} follows from  Lemma \ref{4.1}. Since $\mathscr{N}_{sc}\subseteq \mathscr{N}$ we can infer $c_2\geq c_1>0$. Since
$$
 0\geq\langle J'(u),u^\pm \rangle\geq\langle  J'(u^\pm),u^\pm \rangle,  \ \ \ \forall u\in W_0^{s,p}(\Omega),
$$
 Lemma \ref{1444} provides $\delta_0$ depending only on the parameters such that
\begin{equation}\label{445}
  \ \ \ \ \ \ \ \ \ \  \norm{u^\pm}\geq\delta_0, \ \ \ \forall u\in \mathscr{N}_{sc}.
\end{equation}
Pick a minimizing sequence $\{v_n\}\subseteq \mathscr{N}_{sc}$. Since $\langle J'(v_n),v_n \rangle=0$ and $J(v_n)\rightarrow c_2$, Remark \ref{219} ensures that $\{v_n\}$ is bounded and up to subsequences converges weakly in $W_0^{s,p}(\Omega)$ and strongly in $L^r(\Omega,\frac{dx}{|x|^\alpha})$ for $1\leq r<p_\alpha^*$. First observe that $v^{\pm}\neq0$. Indeed we have
\begin{equation*}\label{446}
\int_{\Omega}\int_{\Omega}\frac{|v_n^\pm(y)|^q|v_n(x)|^q}{|x-y|^\mu}\,dxdy\rightarrow \int_{\Omega}\int_{\Omega}\frac{|v^\pm(y)|^q|v(x)|^q}{|x-y|^\mu}\,dxdy, \ \ \ \  \int_{\Omega}\frac{|v_n^\pm|^r}{|x|^\alpha}\,dx\rightarrow\int_{\Omega}\frac{|v^\pm|^r}{|x|^\alpha}\,dx,
\end{equation*}
and since $\langle J'(v_n),v_n^\pm\rangle=0$ for all $n\in\N$, we deduce from \eqref{1333} and \eqref{445}
\begin{eqnarray}\label{555555}
\begin{aligned}
\lambda\int_{\Omega}\frac{|v^\pm|^r}{|x|^\alpha}\,dx+\gamma\int_{\Omega}\int_{\Omega}\frac{|v^\pm|^q|v|^q}{|x-y|^\mu}\,dxdy&
=\hbox{lim}_{n\to\infty}\left(\int_{\Omega}\lambda\frac{|v_n^\pm|^r}{|x|^\alpha}\,dx+\gamma\int_{\Omega}\int_{\Omega}\frac{|v_n^\pm|^q|v_n|^q}{|x-y|^\mu}\,
dxdy\right)
\\&=\hbox{lim}_{n\to\infty}\langle (-\Delta)_p^sv_n,v_n^\pm \rangle \geq\hbox{lim}_{n\to\infty} \hbox{inf}\norm{v_n^\pm}^p\geq \delta_0^p>0.
\end{aligned}
\end{eqnarray}
Now we claim  that $v\in \mathscr{N}_{sc}$, i.e.
\begin{equation}\label{447}
  \langle J'(v),v^\pm  \rangle=0.
\end{equation}
The functional $u\mapsto\langle (-\Delta)_p^su,u^\pm\rangle$ is weakly sequentially lower semicontinuous by Fatou's lemma, since it can be represented as a nonnegative integral of the form $f(x,y,u(x),u(y))$. Therefore,  $u\mapsto\langle J'(u),u^\pm \rangle$ is wealy sequentially lower semicontinuous, and since \eqref{447} holds for any $v_n$, we deduce $\langle J'(v),v^\pm\rangle\leq0$. Suppose that,  $\langle J'(v),v^+ \rangle<0$ and $\Phi(v)$ be the projection on $\mathscr{N}_{sc}$ of $v$ given by Lemma \ref{4.1}, i.e.
\[
\ \ \ \ \ \ \ \ \ \ \Phi(v)=\bar{t}_+v^++\bar{t}_-v^-, \ \ \ \ (\bar{t}_+,\bar{t}_-)=ArgmaxJ(t_1v^++t_2v^-).
\]
By \eqref{228}, \eqref{229}, $\langle J'(v),v^+ \rangle<0$ and $\langle J'(v),v^- \rangle\leq 0$, it holds $\bar{t}_+<1$ and $\bar{t}_-\leq1$. Since $\langle J'(\Phi(v)),\Phi(v) \rangle=0$, we have
\begin{eqnarray*}
\begin{aligned}
J(\Phi(v))&=(\frac{\lambda}{p}-\frac{\lambda}{r})\int_{\Omega}\frac{|\Phi(v)|^r}{|x|^\alpha}\,dx\,dx+(\frac{\gamma}{p}-\frac{\gamma}{2q})
\int_{\Omega}\int_{\Omega}\frac{|\Phi(v)|^q|\Phi(v)|^q}{|x-y|^\mu}\,dxdy
\\&=(\frac{\lambda}{p}-\frac{\lambda}{r})\left(\bar{t}_+^r\int_{\Omega}\frac{|v^+|^r}{|x|^\alpha}\,dx+\bar{t}_-^r\int_{\Omega}\frac{|v^-|^r}{|x|^\alpha}\,dx\right)
\\&+(\frac{\gamma}{p}-\frac{\gamma}{2q})(\bar{t}_+^{2q}\int_{\Omega}\int_{\Omega}\frac{|v^+(x)|^q|v^+(y)|^q}{|x-y|^\mu}\,dxdy
+\bar{t}_-^{2q}\int_{\Omega}\int_{\Omega}\frac{|v^-(x)|^q|v^-(y)|^q}{|x-y|^\mu}\,dxdy
\\&+2(\bar{t}_+)^q(\bar{t}_-)^q\int_{\Omega}\int_{\Omega}\frac{|v^+(x)|^q|v^-(y)|^q}{|x-y|^\mu}\,dxdy)
\\&<(\frac{\lambda}{p}-\frac{\lambda}{r})\int_{\Omega}\frac{|v|^r}{|x|^\alpha}\,dx+(\frac{\gamma}{p}
-\frac{\gamma}{2q})\int_{\Omega}\int_{\Omega}\frac{|v|^q|v|^q}{|x-y|^\mu}\,dxdy
\\&=\hbox{lim}_{n\to\infty}\left((\frac{\lambda}{p}-\frac{\lambda}{q})\int_{\Omega}\frac{|v_n|^r}{|x|^\alpha}\,dx
+(\frac{\gamma}{p}-\frac{\gamma}{2q})\int_{\Omega}\int_{\Omega}\frac{|v_n|^q|v_n|^q}{|x-y|^\mu}\,dxdy\right)
\\&=\hbox{lim}_{n\to\infty}J(v_n)=c_2,
\end{aligned}
\end{eqnarray*}
which is a contradiction and \eqref{447} holds. This in turn implies that $v\in \mathscr{N}$ and $v_n\rightarrow v$ strong in $W_0^{s,p}(\Omega)$ since
\begin{eqnarray*}\begin{aligned}
\hbox{lim}_{n\to\infty}\norm{v_n}^p&=\hbox{lim}_{n\to\infty}\left(\lambda\int_{\Omega}\frac{|v_n|^r}{|x|^\alpha}\,dx+
\gamma\int_{\Omega}\int_{\Omega}\frac{|v_n(x)|^q|v_n(y)|^q}{|x-y|^\mu}\,dxdy\right)\\
&=\lambda\int_{\Omega}\frac{|v|^r}{|x|^\alpha}\,dx+\gamma\int_{\Omega}\int_{\Omega}\frac{|v(x)|^q|v(y)|^q}{|x-y|^\mu}dxdy=\norm{v}^p,\end{aligned}
\end{eqnarray*}
and therefore $J(v_n)\rightarrow J(v)=c_2$, proving that $v$ solves problem \eqref{336}. \\
Next we prove that $v$ is a critical point for $J$. Suppose by contradiction that $J'(v)\neq0$, then there exits $\varphi\in W_0^{s,p}(\Omega)$ such that $\langle J'(v),\varphi \rangle<-1$ and therefore by continuity there exists a sufficiently small $\delta_0>0$ such that
\begin{equation}\label{449}
 \langle J'(t_1v^++t_2v^-+\delta\varphi),\varphi \rangle<-1, \ \ \ \ \ \ \ if \ |1-t_1|, \ |1-t_2|, \ |\delta|<\delta_0.
\end{equation}
The function $\psi$ in Lemma \ref{4.1} is smooth and strictly concave and has a strict maximum in $(1,1)$, therefore for some $\varepsilon>0$,
$$
 \nabla\psi(s_1,s_2)\neq(0,0), \ \ \ \ \ \ \  \forall \ 0<|(s_1-1,s_2-1)|\leq \varepsilon.
$$
This implies by changing variables $(s_1,s_2)\mapsto(t_1^p,t_2^p)$ that for some $\varepsilon'>0$, which we can suppose smaller than $\delta_0/2$,
\begin{equation}\label{448}
 \big(\langle J'(t_1v^++t_2v^-),v^+ \rangle, \langle J'(t_1v^++t_2v^-),v^-\rangle \big)\neq (0,0), \ \ if \ 0<|(t_1-1,t_2-1)|\leq \varepsilon'
\end{equation}

Let $B=B((1,1),\varepsilon')$ and for any $\delta>0$, $(t_1,t_2)\in B$, define
$$
  u=u(\delta,t_1,t_2)=t_1v^++t_2v^-+\delta\varphi,
$$
(which is continuous in $\delta, t_1,t_2$), noting that for sufficiently small $\delta$, $u^\pm\neq0$. Now consider the field
$$
G_{\delta}(t_1,t_2)=(\langle J'(u),u^+ \rangle,\langle J'(u),u^- \rangle)\in \R^2.
$$
From \eqref{448},  $\mathop{inf}_{\partial B}|G_\delta|>0$ for $\delta=0$, and therefore  the same holds for any sufficiently small $\delta>0$ by continuity. Clearly $\delta\mapsto G_\delta$ is a homotopy, and $G_0(1,1)=(0,0)$. By elementary degree theory, the equation $G_\delta(t_1,t_2)=(0,0)$ has a solution
$(\widetilde{t_1},\widetilde{t_2})\in B$ for some small $\delta\in]0,\delta_0/2[$. If $\bar{v}=\widetilde{t_1}v^++\widetilde{t_2}v^-+\delta\varphi$, this amounts to $\bar{v}\in \mathscr{N}_{sc}$. Since it holds,
$$
J(\bar{v})=J(\widetilde{t_1}v^++\widetilde{t_2}v^-)+\int_0^\delta \langle J'(\widetilde{t_1}v^++\widetilde{t_2}v^-+\delta\varphi),\varphi \rangle\,dt,
$$
\eqref{449} applied to the integral provides
$$
J(\bar{v})\leq J(\widetilde{t_1}v^++\widetilde{t_2}v^-)-\delta.
$$

Since $v\in \mathscr{N}_{sc}$, Theorem \ref{4.1} gives $J(\widetilde{t_1}v^++\widetilde{t_2}v^-)\leq J(v)$, which, inserted into the previous inequality, contradicts the minimality of $J(v)$.
\end{proof}

\begin{remark}
For the proof of case (ii), we only point out the differences. In fact, \eqref{555555} should be
\begin{eqnarray*}\label{66}
\begin{aligned}
\gamma\int_{\Omega}\int_{\Omega}\frac{|v^\pm|^q|v|^q}{|x-y|^\mu}\,dxdy&=
\hbox{lim}_{n\to\infty}\gamma\int_{\Omega}\int_{\Omega}\frac{|v_n^\pm|^q|v_n|^q}{|x-y|^\mu}\,dxdy\\&
\\&=\hbox{lim}_{n\to\infty}\left(\langle (-\Delta)_p^sv_n,v_n^\pm \rangle -\lambda\int_\Omega\frac{|v_n|^p}{|x|^{sp}}dx \right)\\&
\geq\hbox{lim}_{n\to\infty} \hbox{inf}\norm{v_n^\pm}^p(1-\frac{\lambda}{S_{sp,p}})\geq (1-\frac{\lambda}{S_{sp,p}}) \delta_0^p>0.
\end{aligned}
\end{eqnarray*}
So, $v$ is sign-changing. From \cite{MR77777}, $\langle (-\Delta)_p^sv_n,v_n^\pm \rangle-\lambda\int_\Omega\frac{|u^\pm|^p}{|x|^{ps}}dx$ is sequentially
weakly lower semicontinuity, $\langle J'(v),v^\pm\rangle\leq0$, and $v\in\mathscr{N}_{sc}$. From \eqref{2.15}, $v_n\to v$ in $W_0^{s,p}(\Omega)$.
and $J(v)=c_2$ by continuity, also we can get that $v$ is a critical point of $J$.
\end{remark}


\bigskip






\begin{thebibliography}{10}
\bibitem{ACTY} C. O. Alves, D. Cassani, C. Tarsi, M. B. Yang, Existence and concentration of ground state solutions for a critical nonlocal Schrodinger equation in $\R^2$, {\em J. Differential Equations}, 261, 1933-1972, (2016).

\bibitem{AGSY} C. O. Alves, F.S. Gao; M. Squassina, M. B. Yang, Singularly perturbed
critical Choquard equations, {\em J. Differential Equations}, 263, 3943-3988, (2017).

\bibitem{COAABNMY20}
C.~O.~Alves, A.~B.~N¨®brega, M.~Yang,
Multi-bump solutions for Choquard equation with deepening potential well,
\newblock{\em Calc. Var. Partial Differential Equations}, { 55(3)}, 1-28, (2016).


\bibitem{COAMY19}
C.~O.~Alves, M.~Yang,
Existence of semiclassical ground state solutions for a generalized Choquard equation,
\newblock{\em J. Differential Equations}, {257(11)}, 4133-4164, (2014).


\bibitem{VA}
V.~Ambrosio,
Multiplicity and concentration results for a fractional Choquard equation via penalization Method,
\newblock{\em Potential Anal}, {(1)}, 1-28, (2017).

\bibitem{pdSGSM18}
P.~d'Avenia, G.~Siciliano, S.~ Marco,
On fractional Choquard equations,
\newblock{\em Math. Models Methods Appl. Sci.}, { 25(08)}, 1447-1476, (2015).

\bibitem{C} W.J. Chen, Critical fractional p-Kirchhoff type problem with a generalized Choquard nonlinearity,
{\em Journal of Mathematical Physics}, 59, 121502 (2018).

\bibitem{CL}
Y.~H.~Chen, C.~Liu,
Ground state solutions for non-autonomous fractional Choquard equations,
\newblock{\em Nonlinearity }, {29}, 1827-1842, (2016).

\bibitem{MR77777}
W. Chen, S. Mosconi, M. Squassina,
\newblock Nonlocal problems with critical Hardy nonlinearity,
\newblock {\em J. Funct. Anal.}, 275, 3065--3114, (2018).



\bibitem{GSYZ} F.S. Gao, E. Silva, M.B. Yang, J.Z. Zhou, Existence of solutions for critical
 Choquard equations via the concentration compactness method, {\em Proc. Roy. Soc. Edinb.}, 131, (2018).

\bibitem{GFYM123}
F.~Gao, M.~Yang,
On nonlocal Choquard equations with Hardy-Littlewood-Sobolev critical exponents ¡î,
\newblock{\em J. Math. Anal. Appl.}, {448(2)}, 1006-1041, (2017).

\bibitem{GY} F.S. Gao, M.B. Yang, A strongly indefinite Choquard equation with critical
exponent due to the Hardy-Littlewood-Sobolev inequality, {\em Commun.Contemp. Math.}, 20(4), 1750037, (2018).



\bibitem{GY} F.S. Gao, M.B. Yang, The Brezis-Nirenberg type critical problem for
nonlinear Choquard equation, {\em Sci. China Math.} 61(7), 1219-1242, (2018).

\bibitem{GMS} J. Giacomoni, T. Mukherjee and K. Sreenadh, Multiplicity results for Critical
growth Choquard systems, {\em Journal of Mathematical analysis and applications}, 467,
638-672, (2018).

\bibitem{LL}
E. H. Lieb, M. Loss. Analysis, volume 14 of graduate studies in mathematics.
\newblock{\em Amer. Math. Soc.}, Providence, RI, 4, (2001).



\bibitem{MLZL2}
P.~Ma, J.~Zhang,
Existence and multiplicity of solutions for fractional Choquard equations,
\newblock{\em Nonlinear Analysis}, {164}, 100-117, (2017).

\bibitem{MS2017}T. Mukherjee, K. Sreenadh, Positive solutions for nonlinear Choquard equation with singular nonlinearity,{\em
 Complex Variables and Elliptic Equations}, 62(8), 1044-1071, (2017).

\bibitem{MS20181} T. Mukherjee, K. Sreenadh, On Concentration of least energy solutions
for magnetic critical Choquard equations, {\em Journal of Mathematical Analysis and
Applications}, 464(1), 402-420, (2018).

\bibitem{MS20182} T. Mukherjee, K. Sreenadh, On doubly nonlocal p-fractional coupled elliptic
system, {\em Topological methods in nonlinear analysis}, 51(2), 609-636, (2018).


\bibitem{TMKS3}
T.~Mukherjee, K.~Sreenadh,
Fractional Choquard equation with critical nonlinearities,
\newblock {\em Nodea. Nonlinear Diff}, { 24(6)}, 63, (2017).


\bibitem{MR713209}
E.~D.~Nezza, G.~Palatucci, and E.~Valdinoci,
Hitchhiker's guide to the fractional {S}obolev spaces,
\newblock {\em Bull. Sci. Math.}, { 136}, 521--573, (2012).


\bibitem{MR874523}
K.~Perera, M.~Squassina, Y.~Yang,
Bifurcation and multiplicity results for critical fractional $p$-Laplacian problems,
\newblock{\em Math. Nachr.}, {289(2-3)}, 332-342, (2016).

\bibitem{PPMXBZ}
P.~Pucci, M.~Xiang, B.~Zhang,
Existence results for Schr\"odinger-Choquard-Kirchhoff equations involving the fractional $p$-Laplacian,
\newblock{\em Adv. Calc. Var}, (2017).

\bibitem{SGY} Z.F. Shen, F.S. Gao, M.B. Yang, On critical Choquard equation with
potential well, {\em Discrete Contin. Dyn. Syst. A.},138(7), 3669-3695, (2018)

\bibitem{ZSFGMY2}
Z.~Shen, F.~Gao, M.~Yang,
Ground states for nonlinear fractional Choquard equations with general nonlinearities,
\newblock{\em Math. Method. Appl. Sci.}, { 39(14)}, 4082-4098, (2016).

\bibitem{GS5}
G.~Singh,
Nonlocal Pertubations of Fractional Choquard Equation, {\em Advances in Nonlinear Analysis}, (2019).

\bibitem{FWMX4}
F.~Wang, M.~Xiang,
Multiplicity of solutions for a class of fractional Choquard-Kirchhoff equations involving critical nonlinearity,
\newblock{\em Anal. Math. Phys.}, 1-16, (2017).



\bibitem{YWYY}
Y.~Wang, Y.~Yang,
Bifurcation results for the critical Choquard problem involving fractional $p$-Laplacian operator,{\em  Boundary Value Problems}, 132, (2018).

\bibitem{Y} M.B. Yang, Semiclassical ground state solutions for a Choquard type equation in $\R^2$ with critical exponential growth,{\em ESAIM: Control, Optimisation and Calculus of Variations}, 24, 177-209, (2018).

\bibitem{MR88888}
Y. Yang,
\newblock The Brezis-Nirenberg problem for the fractional $p$-Laplacian involving critical Hardy-Sobolen exponents,
\newblock {\em Preprint.} (2018).




\end{thebibliography}
\end{document}